\documentclass[12pt,reqno]{article}

\usepackage[usenames]{color}
\usepackage[colorlinks=true,
linkcolor=webgreen, filecolor=webbrown,
citecolor=webgreen]{hyperref}

\definecolor{webgreen}{rgb}{0,.5,0}
\definecolor{webbrown}{rgb}{.6,0,0}

\usepackage{amssymb}
\usepackage{graphicx}
\usepackage{amscd}
\usepackage{lscape}
\usepackage{tikz}
\usepackage{tikz-cd}
\usepackage{pgfplots}

\usetikzlibrary{matrix}
\usetikzlibrary{fit,shapes}
\usetikzlibrary{positioning, calc}
\tikzset{circle node/.style = {circle,inner sep=1pt,draw, fill=white},
        X node/.style = {fill=white, inner sep=1pt},
        dot node/.style = {circle, draw, inner sep=5pt}
        }
\usepackage{tkz-fct}

\usepackage{amsthm}
\newtheorem{theorem}{Theorem}
\newtheorem{lemma}[theorem]{Lemma}
\newtheorem{proposition}[theorem]{Proposition}
\newtheorem{corollary}[theorem]{Corollary}
\newtheorem{conjecture}[theorem]{Conjecture}

\theoremstyle{definition}

\newtheorem{example}[theorem]{Example}

\usepackage{float}

\usepackage{graphics,amsmath}
\usepackage{amsfonts}
\usepackage{latexsym}
\usepackage{epsf}

\newcommand{\seqnum}[1]{\href{http://oeis.org/#1}{\underline{#1}}}

\begin{document}

\begin{center}
\vskip 1cm{\LARGE\bf On the partial sums of Riordan arrays} \vskip 1cm \large
Paul Barry\\
School of Science\\
Waterford Institute of Technology\\
Ireland\\
\href{mailto:pbarry@wit.ie}{\tt pbarry@wit.ie}
\end{center}
\vskip .2 in

\begin{abstract} We define two notions of partial sums of a Riordan array, corresponding respectively to the partial sums of the rows and the partial sums of the columns of the Riordan array in question. We characterize the matrices that arise from these operations. On the one hand, we obtain a new Riordan array, while on the other hand, we obtain a rectangular array which has an inverse that is a lower Hessenberg matrix. We examine the structure of these Hessenberg matrices. We end with a generalization linked to the Fibonacci numbers and phyllotaxis. \end{abstract}

\section{Introduction} A Riordan array may be defined by a pair of generating functions
$$g(x)=g_0+g_1 x + g_2 x^2+ \cdots,$$ where $g_0 \ne 0$, and
$$f(x)=f_1 x + f_2 x^2+f_3 x^3 +\cdots,$$ where $f_0=0$ and $f_1 \ne 0$. The coefficients $g_i$ and $f_i$ may be drawn from any ring which allows us to carry out the operations that will follow. Often, that ring is one of $\mathbb{Z}$, $\mathbb{R}$ or $\mathbb{C}$. To such a pair of generating functions we may associate the lower-triangular matrix whose $(n,k)$-th element is given by
$$[x^n] g(x) f(x)^k.$$ Here, $[x^n]$ is the functional that extracts the coefficient of $x^n$ in the power series to which it is applied. We use the notation $(g(x), f(x))$ to represent the defining pair of generating functions, and where no confusion can arise, we use it also to denote the corresponding matrix. Such arrays are closed under multiplication, according to the rule
$$(g(x), f(x)) \cdot (u(x), v(x))= (g(x) u(f(x)), v(f(x)),$$ which in terms of the matrix representation of pairs such as $(g(x), f(x))$ corresponds to ordinary matrix multiplication. We can also define the inverse of a Riordan array $(g(x), f(x))$ to be the Riordan array
$$(g(x), f(x))^{-1} =\left(\frac{1}{g(\bar{f}(x))}, \bar{f}(x)\right),$$ where $\bar{f}(x)$ is the compositional inverse of $f(x)$. Thus $\bar{f}(x)$ is the solution $u=u(x)$ of $f(u)=x$ which satisfies $u(0)=0$.
Riordan arrays have an action on power series given by the \emph{fundamental theorem of Riordan arrays}, which describes this action as
$$(g(x), f(x))\cdot h(x)=g(x)h(f(x)).$$
In matrix terms, this corresponds to multiplying the column vector whose elements are given by the expansion of $h(x)$ by the matrix corresponding to $(g(x), f(x))$. 

The bivariate generating function of the Riordan array is given by $\frac{g(x)}{1-yf(x)}$. The row sums $\sum_{k=0}^n [x^n]g(x)f(x)^k$ of a Riordan array have generating function $\frac{g(x)}{1-f(x)}$. The diagonal sums of $(g(x), f(x))$ have generating function $\frac{g(x)}{1-xf(x)}$. 

Many Riordan arrays and many sequences associated with Riordan arrays are to be found in the On-Line Encyclopedia of Integer Sequences (OEIS) \cite{SL1, SL2}. For instance, the Fibonacci numbers $F_{n+1}$ \seqnum{A000045} are given by the diagonal sums of the binomial matrix $\left(\frac{1}{1-x}, \frac{x}{1-x}\right)$, which is \seqnum{A007318}.

We shall define the \emph{row partial sum} of the Riordan array $(g(x), f(x))$ to be the matrix with general $(n,k)$-th element given by
$$\sum_{i=0}^n [x^i]g(x) f(x)^k.$$
\begin{example} We consider the binomial array (Pascal's triangle) given by $\left(\frac{1}{1-x}, \frac{x}{1-x}\right)$.  The corresponding matrix has general term $\binom{n}{k}$ and begins $$\left(
\begin{array}{ccccccc}
 1 & 0 & 0 & 0 & 0 & 0 & 0 \\
 1 & 1 & 0 & 0 & 0 & 0 & 0 \\
 1 & 2 & 1 & 0 & 0 & 0 & 0 \\
 1 & 3 & 3 & 1 & 0 & 0 & 0 \\
 1 & 4 & 6 & 4 & 1 & 0 & 0 \\
 1 & 5 & 10 & 10 & 5 & 1 & 0 \\
 1 & 6 & 15 & 20 & 15 & 6 & 1 \\
\end{array}
\right).$$ 
Thus the row partial sum of this matrix is the matrix with general element
$$\sum_{i=0}^n \binom{i}{k}=\binom{n+1}{n-k}.$$
This matrix begins 
$$\left(
\begin{array}{ccccccc}
 1 & 0 & 0 & 0 & 0 & 0 & 0 \\
 2 & 1 & 0 & 0 & 0 & 0 & 0 \\
 3 & 3 & 1 & 0 & 0 & 0 & 0 \\
 4 & 6 & 4 & 1 & 0 & 0 & 0 \\
 5 & 10 & 10 & 5 & 1 & 0 & 0 \\
 6 & 15 & 20 & 15 & 6 & 1 & 0 \\
 7 & 21 & 35 & 35 & 21 & 7 & 1 \\
\end{array}
\right).$$ 
\end{example}
We define the \emph{column partial sum} of the Riordan array $(g(x), f(x))$ to be the matrix with general $(n,k)$-th element given by
$$\sum_{i=0}^k [x^n] g(x) f(x)^i.$$
\begin{example}
Again, we take the example of the binomial array $\left(\frac{1}{1-x}, \frac{x}{1-x}\right)$.
The general element of the row partial sum in this case is given by $\sum_{i=0}^k \binom{n}{i}$. The resulting matrix begins
$$\left(
\begin{array}{ccccccc}
 1 & 1 & 1 & 1 & 1 & 1 & 1 \\
 1 & 2 & 2 & 2 & 2 & 2 & 2 \\
 1 & 3 & 4 & 4 & 4 & 4 & 4 \\
 1 & 4 & 7 & 8 & 8 & 8 & 8 \\
 1 & 5 & 11 & 15 & 16 & 16 & 16 \\
 1 & 6 & 16 & 26 & 31 & 32 & 32 \\
 1 & 7 & 22 & 42 & 57 & 63 & 64 \\
\end{array}
\right).$$ 
\end{example}
For simplicity, we shall assume that $g_0=1$ and that $f_1=1$ throughout, and furthermore that $g_i, f_i \in \mathbb{Z}$.  Thus we work in the subgroup of the Riordan group for which this is true.

\section{The column partial sum} 
The column partial sum of a Riordan array is easy to analyze. We have the following result.
\begin{proposition} The column partial sum of the Riordan array $(g(x), f(x))$ is the Riordan array 
$$\left(\frac{1}{1-x}, x\right) \cdot (g(x), f(x))= \left(\frac{g(x)}{1-x}, f(x)\right).$$
\end{proposition} 
\begin{proof} 
The Riordan array $\left(\frac{1}{1-x}, x\right)$ is the lower-triangular matrix with $1$s on and below the diagonal, and $0$ elsewhere. If the Riordan array $(g(x), f(x))$ has general term $a_{n,k}=[x^n]g(x)f(x)^k$, then it is clear that the product matrix above has general element 
$$\sum_{i=0}^n 1.a_{i,k}=\sum_{i=0}^n a_{i,k}=\sum_{i=0}^n [x^i]g(x)f(x)^k,$$ as required.
\end{proof}
We shall denote the matrix $\left(\frac{1}{1-x}, x\right)$ by $\Sigma$. Thus the matrix $\Sigma$ begins 
$$\left(
\begin{array}{cccccc}
 1 & 0 & 0 & 0 & 0 & 0 \\
 1 & 1 & 0 & 0 & 0 & 0 \\
 1 & 1 & 1 & 0 & 0 & 0 \\
 1 & 1 & 1 & 1 & 0 & 0 \\
 1 & 1 & 1 & 1 & 1 & 0 \\
 1 & 1 & 1 & 1 & 1 & 1 \\
\end{array}
\right),$$ with inverse $\Sigma^{-1}=(1-x,x)$ which begins 
$$\left(
\begin{array}{cccccc}
 1 & 0 & 0 & 0 & 0 & 0 \\
 -1 & 1 & 0 & 0 & 0 & 0 \\
 0 & -1 & 1 & 0 & 0 & 0 \\
 0 & 0 & -1 & 1 & 0 & 0 \\
 0 & 0 & 0 & -1 & 1 & 0 \\
 0 & 0 & 0 & 0 & -1 & 1 \\
\end{array}
\right).$$ 
The generating function of the column partial sum of $(g(x), f(x))$ is then given by 
$$\frac{g(x)}{(1-x)(1-yf(x))}.$$ 
\begin{corollary}
The row sums and the diagonal sums of the column partial sum of the Riordan array $(g(x), f(x))$ are the partial sums of the row sums and the diagonal sums of the array $(g(x), f(x))$. 
\end{corollary}
\begin{proof} 
The row sums of the column partial sum of $(g(x), f(x))$ has generating function 
$$\frac{g{x}}{(1-x)(1-f(x))} = \frac{1}{1-x} \frac{g(x)}{1-f(x)}.$$ 
The diagonal sums of the column partial sum of $(g(x), f(x))$ has generating function 
$$\frac{g(x)}{(1-x)(1-xf(x))}=\frac{1}{1-x} \frac{g(x)}{1-xf(x)}.$$ 
\end{proof}
\begin{example} The diagonal sums of the binomial matrix are given by Fibonacci numbers $F_{n+1}$. Hence the diagonal sums of the column partial sum of the binomial matrix are given by the sequence $\sum_{k=0}^n F_{k+1}$. \end{example}
\begin{proposition} The inverse of the column partial sum of $(g, f)$ is given by the array 
$$(1-\bar{f},x) \cdot (g, f)^{-1}=(g, f)^{-1} \cdot (1-x, x).$$ 
\end{proposition}
\begin{proof} 
We have 
\begin{align*}\left(\left(\frac{1}{1-x}, x\right) \cdot (g, f)\right)^{-1}&=
(g, f)^{-1} \cdot \left(\frac{1}{1-x}, x\right)^{-1}\\
&=(g, f)^{-1} \cdot (1-x, x)\\
&=\left(\frac{1}{g(\bar{f})}, \bar{f}\right) \cdot (1-x, x)\\
&=\left(\frac{1}{g(\bar{f})}(1-\bar{f}), \bar{f}\right)\\
&=(1-\bar{f},x) \cdot \left(\frac{1}{g(\bar{f})}, \bar{f}\right)\\
&= (1-\bar{f},x) \cdot (g, f)^{-1}.\end{align*}
\end{proof}
The column partial sum of a Riordan array is another Riordan array. This is not the case for the row partial sum.
\section{The row partial sum}
We have the following product characterization of the row partial sum of a Riordan array.
\begin{proposition} The row partial sum of the Riordan array $(g(x), f(x))$ is given by the matrix 
$$(g, f) \cdot \left(\frac{1}{1-x}, x\right)^T.$$ 
\end{proposition}
In this proposition, although we have used Riordan array notation, we note that the transposed matrix $\left(\frac{1}{1-x}, x\right)^T$, which begins 
$$\left(
\begin{array}{ccccccc}
 1 & 1 & 1 & 1 & 1 & 1 & 1 \\
 0 & 1 & 1 & 1 & 1 & 1 & 1 \\
 0 & 0 & 1 & 1 & 1 & 1 & 1 \\
 0 & 0 & 0 & 1 & 1 & 1 & 1 \\
 0 & 0 & 0 & 0 & 1 & 1 & 1 \\
 0 & 0 & 0 & 0 & 0 & 1 & 1 \\
 0 & 0 & 0 & 0 & 0 & 0 & 1 \\
\end{array}
\right),$$ is not a Riordan array. 
\begin{proof} The $(n,k)$-th element of the row partial sum of $(g(x), f(x))$ is given by 
$$\sum_{i=0}^k [x^n] g(x) f(x)^i = \sum_{i=0}^k ([x^n]g(x)f(x)^i).1.$$ 
\end{proof} 
\begin{corollary} The generating function of the row partial sum of $(g(x), f(x))$ is given by 
$$\frac{1}{1-y} \frac{g(x)}{1-yf(x)}=\frac{g(x)}{(1-y)(1-yf(x))}.$$ 
\end{corollary}
This is a translation of the fact that we are summing rows in the ``$y$-direction''. 
\begin{proof} 
We have 
\begin{align*} 
[y^k][x^n] \frac{1}{1-y} \frac{g(x)}{1-yf(x)} &= [y^k] \left(\sum_{j=0}^{\infty} y^j\right)\left(\sum_{i=0}^{\infty} y^i[x^n]g(x)f(x)^i\right)\\
&= [y^k] \sum_j \sum_i y^{i+j} [x^n]g(x) f(x)^i \\
&= \sum_{i=0}^k  [x^n]g(x)f(x)^i.\end{align*}
\end{proof}
In order to study the inverse of the row partial sum of a Riordan array, we need to know what the inverse of $\left(\frac{1}{1-x}, x\right)^T$ is. Since the inverse of a transposed matrix is the transpose of the inverse, we can in fact use Riordan arrays to describe this inverse.
\begin{lemma} The inverse of the matrix $\left(\frac{1}{1-x}, x\right)^T$ is given by $(1-x, x)^T$. 
\end{lemma} 
We note that the matrix $(1-x, x)^T$ begins 
$$\left(
\begin{array}{ccccccc}
 1 & -1 & 0 & 0 & 0 & 0 & 0 \\
 0 & 1 & -1 & 0 & 0 & 0 & 0 \\
 0 & 0 & 1 & -1 & 0 & 0 & 0 \\
 0 & 0 & 0 & 1 & -1 & 0 & 0 \\
 0 & 0 & 0 & 0 & 1 & -1 & 0 \\
 0 & 0 & 0 & 0 & 0 & 1 & -1 \\
 0 & 0 & 0 & 0 & 0 & 0 & 1 \\
\end{array}
\right).$$ 
This is Hessenberg in form. The product of a Hessenberg matrix by a lower-triangular matrix (on the left or on the right) is again Hessenberg. Thus we have the following result.
\begin{proposition} The inverse of the row partial sum of a Riordan array always exists and it is a Hessenberg matrix.
\end{proposition}
\begin{proof} The inverse in question is given by 
\begin{align*} \left((g, f) \cdot \left(\frac{1}{1-x}, x\right)^T\right)^{-1}
&= \left(\left(\frac{1}{1-x}, x \right)^T\right)^{-1} \cdot (g, f)^{-1}\\
&= (1-x, x)^T \cdot (g, f)^{-1}.\end{align*}
\end{proof} 
In order to better understand the structure of this inverse, we have the following proposition. In its statement, we use the notation $\overline{M}$ to denote the matrix $M$ with its top row removed. For infinite extent matrices, we have 
$$\overline{M}=UM,$$ where $U$ is the matrix that begins
$$\left(
\begin{array}{ccccccc}
 0 & 1 & 0 & 0 & 0 & 0 & 0 \\
 0 & 0 & 1 & 0 & 0 & 0 & 0 \\
 0 & 0 & 0 & 1 & 0 & 0 & 0 \\
 0 & 0 & 0 & 0 & 1 & 0 & 0 \\
 0 & 0 & 0 & 0 & 0 & 1 & 0 \\
 0 & 0 & 0 & 0 & 0 & 0 & 1 \\
 0 & 0 & 0 & 0 & 0 & 0 & 0 \\
\end{array}
\right).$$ 
For finite matrices, adjustments must be made to the last row of $UM$. 
\begin{example} We consider the matrix $M=\left(\frac{1+x}{1-2x}, \frac{x(1-x)}{1-3x}\right)$ which begins
$$\left(
\begin{array}{cccccc}
 1 & 0 & 0 & 0 & 0 & 0 \\
 3 & 1 & 0 & 0 & 0 & 0 \\
 6 & 5 & 1 & 0 & 0 & 0 \\
 12 & 18 & 7 & 1 & 0 & 0 \\
 24 & 60 & 34 & 9 & 1 & 0 \\
 48 & 192 & 144 & 54 & 11 & 1 \\
\end{array}
\right).$$ 
Its inverse begins 
$$\left(
\begin{array}{cccccc}
 1 & 0 & 0 & 0 & 0 & 0 \\
 -3 & 1 & 0 & 0 & 0 & 0 \\
 9 & -5 & 1 & 0 & 0 & 0 \\
 -21 & 17 & -7 & 1 & 0 & 0 \\
 39 & -43 & 29 & -9 & 1 & 0 \\
 -63 & 83 & -85 & 45 & -11 & 1 \\
\end{array}
\right).$$ 
Its row partial sum begins 
$$\left(
\begin{array}{cccccc}
 1 & 1 & 1 & 1 & 1 & 1 \\
 3 & 4 & 4 & 4 & 4 & 4 \\
 6 & 11 & 12 & 12 & 12 & 12 \\
 12 & 30 & 37 & 38 & 38 & 38 \\
 24 & 84 & 118 & 127 & 128 & 128 \\
 48 & 240 & 384 & 438 & 449 & 450 \\
\end{array}
\right).$$ This has an inverse that begins 
$$\left(
\begin{array}{cccccc}
 4 & -1 & 0 & 0 & 0 & 0 \\
 -12 & 6 & -1 & 0 & 0 & 0 \\
 30 & -22 & 8 & -1 & 0 & 0 \\
 -60 & 60 & -36 & 10 & -1 & 0 \\
 102 & -126 & 114 & -54 & 12 & -1 \\
 -63 & 83 & -85 & 45 & -11 & 1 \\
\end{array}
\right).$$ 
This is equal to 
$$(I+\left(
\begin{array}{ccccccc}
 0 & -1 & 0 & 0 & 0 & 0 & \cdots\\
 0 & 0 & -1 & 0 & 0 & 0 & \cdots\\
 0 & 0 & 0 & -1 & 0 & 0 & \cdots\\
 0 & 0 & 0 & 0 & -1 & 0 & \cdots\\
 0 & 0 & 0 & 0 & 0 & -1 & \cdots\\
 0 & 0 & 0 & 0 & 0 & 0 & \cdots\\
 \vdots & \vdots & \vdots & \vdots & \vdots & \vdots & \ddots\\
\end{array}
\right))M^{-1}$$ which (using a finite truncation) begins 
$$\left(
\begin{array}{ccccccc}
 1 & 0 & 0 & 0 & 0 & 0 \\
 -3 & 1 & 0 & 0 & 0 & 0 \\
 9 & -5 & 1 & 0 & 0 & 0 \\
 -21 & 17 & -7 & 1 & 0 & 0 \\
 39 & -43 & 29 & -9 & 1 & 0 \\
 -63 & 83 & -85 & 45 & -11 & 1 \\
\end{array}
\right)+\left(
\begin{array}{cccccc}
 3 & -1 & 0 & 0 & 0 & 0 \\
 -9 & 5 & -1 & 0 & 0 & 0 \\
 21 & -17 & 7 & -1 & 0 & 0 \\
 -39 & 43 & -29 & 9 & -1 & 0 \\
 63 & -83 & 85 & -45 & 11 & -1 \\
 0 & 0 & 0 & 0 & 0 & 0 \\
\end{array}
\right).$$ 
\end{example}
\begin{proposition}\label{ref} We have 
$$H=\left( (g(x), f(x)) \cdot \left(\frac{1}{1-x}, x \right)^T \right)^{-1}=\overline{\left(\frac{-g(x)}{1-f(x)}, f(x)\right)^{-1}}.$$
\end{proposition} 
\begin{proof} We introduce an auxiliary matrix $V$ which begins 
$$\left(
\begin{array}{ccccccc}
 1 & 1 & 1 & 1 & 1 & 1 & 1 \\
 -1 & 0 & 0 & 0 & 0 & 0 & 0 \\
 0 & -1 & 0 & 0 & 0 & 0 & 0 \\
 0 & 0 & -1 & 0 & 0 & 0 & 0 \\
 0 & 0 & 0 & -1 & 0 & 0 & 0 \\
 0 & 0 & 0 & 0 & -1 & 0 & 0 \\
 0 & 0 & 0 & 0 & 0 & -1 & 0 \\
\end{array}
\right).$$ 
The inverse of this matrix then begins (in the finite case)
$$\left(
\begin{array}{ccccccc}
 0 & -1 & 0 & 0 & 0 & 0 & 0 \\
 0 & 0 & -1 & 0 & 0 & 0 & 0 \\
 0 & 0 & 0 & -1 & 0 & 0 & 0 \\
 0 & 0 & 0 & 0 & -1 & 0 & 0 \\
 0 & 0 & 0 & 0 & 0 & -1 & 0 \\
 0 & 0 & 0 & 0 & 0 & 0 & -1 \\
 1 & 1 & 1 & 1 & 1 & 1 & 1 \\
\end{array}
\right).$$ 
We have 
$$\left(\frac{1}{1-x}, x\right)\cdot V = \left(\frac{1}{1-x}, x\right)^T.$$ 
Thus 
\begin{align*}(g, f) \cdot \left(\frac{1}{1-x}, x\right)^T&= (g, f) \cdot \left(\frac{1}{1-x}, x\right)\cdot V\\
&= \left(\frac{g}{1-f}, f\right) \cdot V.\end{align*}
Then we obtain 
\begin{align*} H&= \left(\left(\frac{g}{1-f}, f\right) \cdot V\right)^{-1}\\
&= V^{-1} \cdot \left(\frac{g}{1-f}, f\right)^{-1} \\
&= \overline{\left(\frac{-g}{1-f}, f\right)^{-1}}.\end{align*}
\end{proof}
Thus the inverse of the row partial sum of a Riordan array is another Riordan array, from which the first row has been removed. We note the following.
\begin{align*} \left(\frac{-g}{1-f}, f\right)^{-1}&=\left(\frac{1}{\frac{-g(\bar{f})}{1-f(\bar{f})}}, \bar{f}\right)\\
&=\left(\frac{1-x}{-g(\bar{f})}, \bar{f}\right)\\
&=(1-x, x)\cdot (-g, f)^{-1}.\end{align*} 
Thus $$H=U \cdot (1-x, x) \cdot (-g, f)^{-1}.$$ The product $U \cdot (1-x, x)$ begins 
$$\left(
\begin{array}{ccccccc}
 -1 & 1 & 0 & 0 & 0 & 0 & 0 \\
 0 & -1 & 1 & 0 & 0 & 0 & 0 \\
 0 & 0 & -1 & 1 & 0 & 0 & 0 \\
 0 & 0 & 0 & -1 & 1 & 0 & 0 \\
 0 & 0 & 0 & 0 & -1 & 1 & 0 \\
 0 & 0 & 0 & 0 & 0 & -1 & 1 \\
 0 & 0 & 0 & 0 & 0 & 0 & -1 \\
\end{array}
\right).$$ 
\begin{example} We consider the binomial matrix $\left(\binom{n}{k}\right)$ corresponding to $\left(\frac{1}{1-x}, \frac{x}{1-x}\right)$. Here, $g(x)=\frac{1}{1-x}$ and $f(x)=\frac{x}{1-x}$. We obtain that $\bar{f}=\frac{x}{1+x}$ and 
$$(-g, f)^{-1}=\left(\frac{-1}{1+x}, \frac{x}{1+x}\right).$$ 
Then $$(1-x, x) \cdot (-g, f)^{-1}= \left(\frac{-(1-x)}{1+x}, \frac{x}{1+x}\right).$$ 
Thus $$\left(\sum_{i=0}^k \binom{n}{i} \right)^{-1}=\overline{\left(\frac{-(1-x)}{1+x}, \frac{x}{1+x}\right)}.$$ 
It is important to note that care must be taken when working with finite truncations of Riordan arrays and their row partial sums. We remain with the case of the binomial matrix, but we take the example of the $6 \times 6$ matrix $$\left(
\begin{array}{cccccc}
 1 & 0 & 0 & 0 & 0 & 0 \\
 1 & 1 & 0 & 0 & 0 & 0 \\
 1 & 2 & 1 & 0 & 0 & 0 \\
 1 & 3 & 3 & 1 & 0 & 0 \\
 1 & 4 & 6 & 4 & 1 & 0 \\
 1 & 5 & 10 & 10 & 5 & 1 \\
\end{array}
\right).$$ 
The (finite) row partial sum of this matrix is 
$$\left(
\begin{array}{cccccc}
 1 & 1 & 1 & 1 & 1 & 1 \\
 1 & 2 & 2 & 2 & 2 & 2 \\
 1 & 3 & 4 & 4 & 4 & 4 \\
 1 & 4 & 7 & 8 & 8 & 8 \\
 1 & 5 & 11 & 15 & 16 & 16 \\
 1 & 6 & 16 & 26 & 31 & 32 \\
\end{array}
\right),$$ with inverse 
$$\left(
\begin{array}{cccccc}
 2 & -1 & 0 & 0 & 0 & 0 \\
 -2 & 3 & -1 & 0 & 0 & 0 \\
 2 & -5 & 4 & -1 & 0 & 0 \\
 -2 & 7 & -9 & 5 & -1 & 0 \\
 2 & -9 & 16 & -14 & 6 & -1 \\
 -1 & 5 & -10 & 10 & -5 & 1 \\
\end{array}
\right).$$ 
Now the $6 \times 6$ truncation of $\overline{\left(\frac{-(1-x)}{1+x}, \frac{x}{1+x}\right)}$ is 
$$\left(
\begin{array}{cccccc}
 2 & -1 & 0 & 0 & 0 & 0 \\
 -2 & 3 & -1 & 0 & 0 & 0 \\
 2 & -5 & 4 & -1 & 0 & 0 \\
 -2 & 7 & -9 & 5 & -1 & 0 \\
 2 & -9 & 16 & -14 & 6 & -1 \\
 -2 & 11 & -25 & 30 & -20 & 7 \\
\end{array}
\right).$$
In the last but one matrix above, the final row $(-1 , 5 , -10 , 10 , -5 , 1)$ is the $6$-th row of the inverse binomial matrix. Subtracting these two last matrices yields the matrix
$$\left(
\begin{array}{cccccc}
 0 & 0 & 0 & 0 & 0 & 0 \\
 0 & 0 & 0 & 0 & 0 & 0 \\
 0 & 0 & 0 & 0 & 0 & 0 \\
 0 & 0 & 0 & 0 & 0 & 0 \\
 0 & 0 & 0 & 0 & 0 & 0 \\
 1 & -6 & 15 & -20 & 15 & -6 \\
\end{array}
\right),$$ where the last row is essentially the $7$-th row of the inverse binomial matrix. 
This example indicates the adjustments that need to be made in the finite case.
\end{example}
We can describe the generating function of $H$ as follows.
\begin{corollary} Let $H$ denote the inverse of the row partial sum of the Riordan array $(g(x), f(x))$. 
Then the bivariate generating function of $H$ is given by
$$\left(\frac{\frac{1-x}{-g(\bar{f}(x))}}{1-y \bar{f}(x)}-1\right)/x=\frac{1}{x}\left(\frac{1-x}{-g(\bar{f}(x))(1-y\bar{f}(x))}-1\right).$$
\end{corollary}
We end this section by looking at the combination of row and column partial sums.
\begin{proposition} The column partial sum of a row partial sum of a Riordan array $M=(g(x), f(x))$ is equal to the row partial sum of the column partial sum of $M$. 
\end{proposition} 
\begin{proof} This follows from the associativity of matrix multiplication, since we have 
$$ \Sigma\cdot(M \cdot \Sigma^T)=\Sigma \cdot M \cdot \Sigma^T = (\Sigma \cdot M)\cdot \Sigma^T.$$
\end{proof} 
\begin{proposition} For a Riordan array $M=(g(x), f(x))$, the diagonal sums of the column partial sum and the diagonal sums of the row partial sum are equal.
\end{proposition}
\begin{proof} The generating function of the column partial sum of $M$ is given by $\frac{1}{1-x} \frac{g(x)}{1-yf(x)}$. The generating function of the row partial sum of $M$ is given by 
$\frac{g(x)}{1-yf(x)} \frac{1}{1-y}$. Setting $y=x$ in both gives the generating function of their respective diagonal sums. We conclude that they are equal.
\end{proof}
\section{More Hessenberg structure}
In this section, we look more closely at the structures inherent in $H$ and its inverse. Thus we are reversing the direction of our study, going from $H$ to its inverse, the row partial sums of $(g(x), f(x))$. First Ikebe \cite{Ikebe} and then Zhong \cite{Zhong} characterized the inverse of a general Hessenberg matrix of order $n$ \cite{Russian}.  Thus let $H$ denote such a matrix.
$$H=\left(
\begin{array}{cccccc}
 h_{1,1} & \alpha_1 & 0 & 0 & \cdots & 0 \\
 h_{2,1} & h_{2,2} & \alpha_2 & 0 & \cdots & 0 \\
 \cdots &  &  &  &  &  \\
 h_{n,1} & h_{n,2} & h_{n,3} & h_{n,4} & \cdots & h_{n,n} \\
\end{array}
\right).$$
We can partition the matrix $H$ into the form 
$$H=\left(
\begin{array}{cc}
 C_{n-1} & P_{n-1} \\
 h_{n,1} & R_{n-1}^T \\
\end{array}
\right),$$ 
where 
$$C_{n-1}=(h_{1,1}, h_{2,1}, \ldots, h_{n-1,1})^T,$$ 
$$R_{n-1}^T=(h_{n,2}, h_{n,3}, \ldots, h_{n,n}),$$ and 
$$P_{n-1}=\left(
\begin{array}{ccccc}
 \alpha_1 & 0 & 0 & \cdots & 0 \\
 h_{2,2} & \alpha_2 & 0 & \cdots & 0 \\
 \cdots &  &  &  &  \\
 h_{n-1,2} & h_{n-1,3} & h_{n-1,4} & \cdots & \alpha_{n-1} \\
\end{array}
\right).$$ 
We then have the following theorem \cite{Zhong}.
\begin{theorem} If $\alpha_i \ne 0\,\, (i=1,2,\ldots,n-1)$ and the matrix $H$ is nonsingular, then 
\begin{equation}\label{eqn}H^{-1}=\left(
\begin{array}{cc}
 0 & 0 \\
 P_{n-1}^{-1} & 0 \\
\end{array}
\right)+(x_1, x_2, \cdots, x_n)^T (w_1, w_2,\cdots,w_n),\end{equation} 
where $x_i\,\, (i=1,2,\ldots,n)$ are defined by 
\begin{align*} x_1 &=1\\
x_i&=-\frac{\sum_{j=1}^{i-1} h_{i-1,j}x_j}{\alpha_{i-1}},\quad i=2,3,\ldots,n,\end{align*}
and $w_i\,\, (i=1,2,\ldots,n)$ can be defined recursively as 
\begin{align*} w_n &=\frac{1}{\sum_{j=1}^n h_{n,j}x_j}\\
w_i&=-\frac{\sum_{j=i+1}^n h_{j,i+1}w_j}{\alpha_i}, \quad i=n-1,\ldots,1.\end{align*}
\end{theorem}
Note that we have $$\sum_{k=1}^n h_{n,k}x_k=\frac{(-1)^{n-1} \det(H)}{\prod_{j=1}^{n-1} \alpha_j}.$$ 
We now wish to relate the foregoing to the case of the row partial sum of the Riordan array $(g(x), f(x))$. Let 
$(g(x), f(x))_n$ denote the $n \times n$ truncation of the Riordan array $(g(x), f(x))$. We let $H$ denote the inverse of the row partial sum of $(g(x), f(x))_n$. This is a Hessenberg matrix, equal to the $n \times n$ truncation of $\overline{\left(\frac{-g}{1-f}, f\right)^{-1}},$ except that the $n$-th row is equal to the $n$-th row of $(g(x), f(x))^{-1}$. 
\begin{proposition} We have $$P_{n-1}=\left(\frac{-gf/x}{1-f}, f\right)_{n-1}^{-1}.$$ 
\end{proposition}
\begin{proof} This follows from Proposition \textbf{\ref{ref}}.
\end{proof}
\begin{corollary} We have 
$$P_{n-1}^{-1} = \left(\frac{-gf/x}{1-f}, f\right)_{n-1}.$$
\end{corollary}
In fact, Eq. (\textbf{\ref{eqn}}), though it refers only to the order $n$ case, leads to a canonical decomposition of the row partial sum of a Riordan matrix. Thus we have the following result.
\begin{proposition} The generating function of the row partial sum of the Riordan array $(g(x), f(x))$ can be expressed as the following sum:
$$\frac{g(x)}{(1-y)(1-yf(x))}=\frac{\frac{-g(x)f(x)}{1-f(x)}}{1-y f(x)}+\frac{1}{1-y}\frac{g(x)}{1-f(x)}.$$
\end{proposition}
\begin{proof} Algebraic manipulation shows that both sides are equal.
\end{proof}
In the truncated case, the generating function $\frac{1}{1-y}\frac{g(x)}{1-f(x)}$ expands to give a matrix all of whose columns are the row sums of $(g(x), f(x))_n$, thus allowing us to identify the vector $(x_1, x_2,  \ldots)$ of the Theorem as the row sum vector of $(g(x), f(x))_n$, while the vector $(w_1,w_2,\ldots)$ is the all $1$'s vector. 
\begin{example}
We take the case of the Riordan array $\left(\frac{1+x}{1-2x}, \frac{x(1-x)}{1-3x}\right)$. This array begins 
$$\left(
\begin{array}{cccccc}
 1 & 0 & 0 & 0 & 0 & 0 \\
 3 & 1 & 0 & 0 & 0 & 0 \\
 6 & 5 & 1 & 0 & 0 & 0 \\
 12 & 18 & 7 & 1 & 0 & 0 \\
 24 & 60 & 34 & 9 & 1 & 0 \\
 48 & 192 & 144 & 54 & 11 & 1 \\
\end{array}
\right),$$ with row sums that begin 
$$1, 4, 12, 38, 128, 450, 1624,\ldots.$$ The corresponding row partial sum matrix begins 
$$S=\left(
\begin{array}{cccccc}
 1 & 1 & 1 & 1 & 1 & 1 \\
 3 & 4 & 4 & 4 & 4 & 4 \\
 6 & 11 & 12 & 12 & 12 & 12 \\
 12 & 30 & 37 & 38 & 38 & 38 \\
 24 & 84 & 118 & 127 & 128 & 128 \\
 48 & 240 & 384 & 438 & 449 & 450 \\
\end{array}
\right).$$ The Hessenberg form inverse of this matrix then begins 
$$\left(
\begin{array}{cccccc}
 4 & -1 & 0 & 0 & 0 & 0 \\
 -12 & 6 & -1 & 0 & 0 & 0 \\
 30 & -22 & 8 & -1 & 0 & 0 \\
 -60 & 60 & -36 & 10 & -1 & 0 \\
 102 & -126 & 114 & -54 & 12 & -1 \\
 -63 & 83 & -85 & 45 & -11 & 1 \\
\end{array}
\right).$$ 
Note that 
$$\left(
\begin{array}{ccccc}
 -1 & 0 & 0 & 0 & 0 \\
 -6 & -1 & 0 & 0 & 0 \\
 -26 & -8 & -1 & 0 & 0 \\
 -104 & -44 & -10 & -1 & 0 \\
 -402 & -210 & -66 & -12 & -1 \\
\end{array}
\right)^{-1}=\left(
\begin{array}{ccccc}
 -1 & 0 & 0 & 0 & 0 \\
 6 & -1 & 0 & 0 & 0 \\
 -22 & 8 & -1 & 0 & 0 \\
 60 & -36 & 10 & -1 & 0 \\
 -126 & 114 & -54 & 12 & -1 \\
\end{array}
\right).$$ 
We then have 
\begin{align*}S&=\left(
\begin{array}{cccccc}
 0 & 0 & 0 & 0 & 0 & 0 \\
 -1 & 0 & 0 & 0 & 0 & 0 \\
 -6 & -1 & 0 & 0 & 0 & 0 \\
 -26 & -8 & -1 & 0 & 0 & 0 \\
 -104 & -44 & -10 & -1 & 0 & 0 \\
 -402 & -210 & -66 & -12 & -1 & 0 \\
\end{array}
\right)+\left(
\begin{array}{cccccc}
 1 & 1 & 1 & 1 & 1 & 1 \\
 4 & 4 & 4 & 4 & 4 & 4 \\
 12 & 12 & 12 & 12 & 12 & 12 \\
 38 & 38 & 38 & 38 & 38 & 38 \\
 128 & 128 & 128 & 128 & 128 & 128 \\
 450 & 450 & 450 & 450 & 450 & 450 \\
\end{array}
\right) \\
&=\left(\begin{array}{cccccc}
 0 & 0 & 0 & 0 & 0 & 0 \\
 -1 & 0 & 0 & 0 & 0 & 0 \\
 -6 & -1 & 0 & 0 & 0 & 0 \\
 -26 & -8 & -1 & 0 & 0 & 0 \\
 -104 & -44 & -10 & -1 & 0 & 0 \\
 -402 & -210 & -66 & -12 & -1 & 0 \\
\end{array}
\right)+\left(
\begin{array}{c}
 1 \\
 4 \\
 12 \\
 38 \\
 128 \\
 450 \\
\end{array}
\right)\left(
\begin{array}{cccccc}
 1 & 1 & 1 & 1 & 1 & 1\\
\end{array}
\right).\end{align*}

\end{example}
There is a corresponding decomposition of the original Riordan array, obtained algebraically by multiplying the bivariate generating functions above by $1-y$. This yields 
$$\frac{g(x)}{1-yf(x)}=\frac{-g(x)f(x)}{1-yf(x)}\frac{1-y}{1-f(x)}+\frac{g(x)}{1-f(x)}.$$
\begin{example}
Returning to the Riordan array $\left(\frac{1+x}{1-2x}, \frac{x(1-x)}{1-3x}\right)$ of the last example, we obtain the decomposition
$$\left(
\begin{array}{cccccc}
 1 & 0 & 0 & 0 & 0 & 0 \\
 3 & 1 & 0 & 0 & 0 & 0 \\
 6 & 5 & 1 & 0 & 0 & 0 \\
 12 & 18 & 7 & 1 & 0 & 0 \\
 24 & 60 & 34 & 9 & 1 & 0 \\
 48 & 192 & 144 & 54 & 11 & 1 \\
\end{array}
\right)=\left(
\begin{array}{cccccc}
 0 & 0 & 0 & 0 & 0 & 0 \\
 -1 & 1 & 0 & 0 & 0 & 0 \\
 -6 & 5 & 1 & 0 & 0 & 0 \\
 -26 & 18 & 7 & 1 & 0 & 0 \\
 -104 & 60 & 34 & 9 & 1 & 0 \\
 -402 & 192 & 144 & 54 & 11 & 1 \\
\end{array}
\right)+\left(
\begin{array}{cccccc}
 1 & 0 & 0 & 0 & 0 & 0 \\
 4 & 0 & 0 & 0 & 0 & 0 \\
 12 & 0 & 0 & 0 & 0 & 0 \\
 38 & 0 & 0 & 0 & 0 & 0 \\
 128 & 0 & 0 & 0 & 0 & 0 \\
 450 & 0 & 0 & 0 & 0 & 0 \\
\end{array}
\right).$$
\end{example}

\section{A production matrix}
One characterization of Riordan arrays is by means of their production matrices \cite{ProdMat}. Given a Riordan array $M$, the matrix
$$P_M= M^{-1} \overline{M}$$ is called the production (or Stieltjes) matrix of $M$. Note that we have 
$$P_{M^{-1}}= M \overline{M^{-1}}.$$ 
A lower triangular matrix is then a Riordan array if and only if its production matrix is such that its columns, starting with the second one, are shifted versions of the same column vector, with this column vector having a non-zero first element. For a Riordan array $M=(g(x), f(x))$, we have that the second column of its production matrix $P_M$ is the expansion of the generating function $\frac{x}{\bar{f}(x)}$. In like manner, the second column of $P_{M^{-1}}$ is given by the expansion of $\frac{x}{f(x)}$. The first column of $P_{M^{-1}}$ has generating function $\frac{1}{f(x)}\left(1-\frac{1}{g(x)}\right)$. 
\begin{example} We consider $M=(g(x), f(x))=\left(\frac{1+x}{1-2x}, \frac{x(1-x)}{1-3x}\right)$. The production matrix of $M^{-1}$ begins
$$\left(
\begin{array}{cccccc}
 -3 & 1 & 0 & 0 & 0 & 0 \\
 0 & -2 & 1 & 0 & 0 & 0 \\
 6 & -2 & -2 & 1 & 0 & 0 \\
 18 & -2 & -2 & -2 & 1 & 0 \\
 42 & -2 & -2 & -2 & -2 & 1 \\
 90 & -2 & -2 & -2 & -2 & -2 \\
\end{array}
\right).$$ 
The sequence $1,-2,-2,\ldots$ has its generating function given by $\frac{x}{f(x)}=\frac{1-3x}{1-x}$. 
\end{example} 
 Returning to the general case, we let $\Sigma$ denote the matrix $\left(\frac{1}{1-x},x\right)$ and let $M=(g(x), f(x))$. We wish to consider the matrix product 
$$M \cdot (\Sigma\cdot M \cdot \Sigma^T)^{-1}= M \cdot (\Sigma^T)^{-1} \cdot M^{-1} \cdot \Sigma^{-1}.$$ 
The matrix $(\Sigma^T)^{-1}$ or $(1-x,x)^T$ begins 
$$\left(
\begin{array}{cccccc}
 1 & -1 & 0 & 0 & 0 & 0 \\
 0 & 1 & -1 & 0 & 0 & 0 \\
 0 & 0 & 1 & -1 & 0 & 0 \\
 0 & 0 & 0 & 1 & -1 & 0 \\
 0 & 0 & 0 & 0 & 1 & -1 \\
 0 & 0 & 0 & 0 & 0 & 1 \\
\end{array}
\right),$$ and so we have 
$$(\Sigma^T)^{-1} \cdot M^{-1}=M^{-1}-\overline{M^{-1}}.$$ 
Thus the four-fold product is equal to 
\begin{align*}M\cdot (M^{-1}-\overline{M^{-1}})\cdot \Sigma^{-1}&=(I-P_{M^{-1}}) \cdot \Sigma^{-1}\\
&=\Sigma^{-1}-P_{M^{-1}}\cdot \Sigma^{-1}.\end{align*}
In the event that $P_{M^{-1}}$ begins 
$$\left(
\begin{array}{cccccc}
 1 & 1 & 0 & 0 & 0 & 0 \\
 \alpha  & r & 1 & 0 & 0 & 0 \\
 \beta  & s & r & 1 & 0 & 0 \\
 \gamma  & t & s & r & 1 & 0 \\
 \delta  & u & t & s & r & 1 \\
 \epsilon  & v & u & t & s & r \\
\end{array}
\right),$$ then $\Sigma^{-1}-P_{M^{-1}}\cdot \Sigma^{-1}$ begins 
$$\left(
\begin{array}{cccccc}
 1 & -1 & 0 & 0 & 0 & 0 \\
 -\alpha +r-1 & 2-r & -1 & 0 & 0 & 0 \\
 s-\beta  & r-s-1 & 2-r & -1 & 0 & 0 \\
 t-\gamma  & s-t & r-s-1 & 2-r & -1 & 0 \\
 u-\delta  & t-u & s-t & r-s-1 & 2-r & -1 \\
 v-\epsilon  & u-v & t-u & s-t & r-s-1 & 1-r \\
\end{array}
\right).$$ We see that this is again the production matrix of a Riordan array. 
\begin{proposition} Let $M$ be the Riordan array $(g(x), f(x))$. Then the Hessenberg matrix 
 $$M \cdot (\Sigma\cdot M \cdot \Sigma^T)^{-1}$$ is the production matrix of the Riordan array 
 $$\left(\frac{1}{1-x} \frac{g}{1-f}, \frac{1}{1-x} \frac{f}{1-f}\right)^{-1}.$$
\end{proposition}
\begin{proof} If $M=(g, f)$ is such that $P_{M^{-1}}$ has the form above, then we have, for instance, 
$$f(x)=x-rx^2+(r^2-s)x^3-(r^3-2rs+t)x^4+\cdots.$$ 
The matrix $\left(\frac{1}{1-x} \frac{g}{1-f}, \frac{1}{1-x} \frac{f}{1-f}\right)^{-1}$ will have 
$$A(x)=\frac{x}{\frac{1}{1-x} \frac{f}{1-f}}.$$
This  expands to give a sequence that begins,
$$1, r - 2, -r + s + 1, t - s, u - t, v - u,\ldots,$$ 
as required.
Similarly for $Z(x)$. 
\end{proof}
\section{Whitney numbers}
We have seen that for the binomial matrix $\left(\binom{n}{k}\right)=\left(\frac{1}{1-x}, \frac{x}{1-x}\right)$ the row partial sum has generating function 
$$\frac{1}{(1-y)(1-x-xy)}.$$ 
The row partial sum begins \cite{Price}
$$\left(
\begin{array}{cccccc}
 1 & 1 & 1 & 1 & 1 & 1 \\
 1 & 2 & 2 & 2 & 2 & 2 \\
 1 & 3 & 4 & 4 & 4 & 4 \\
 1 & 4 & 7 & 8 & 8 & 8 \\
 1 & 5 & 11 & 15 & 16 & 16 \\
 1 & 6 & 16 & 26 & 31 & 32 \\
\end{array}
\right).$$ 
The transposed matrix, with generating function 
$$\frac{1}{(1-x)(1-y-xy)},$$ begins 
$$\left(
\begin{array}{cccccc}
 1 & 1 & 1 & 1 & 1 & 1 \\
 1 & 2 & 3 & 4 & 5 & 6 \\
 1 & 2 & 4 & 7 & 11 & 16 \\
 1 & 2 & 4 & 8 & 15 & 26 \\
 1 & 2 & 4 & 8 & 16 & 31 \\
 1 & 2 & 4 & 8 & 16 & 32 \\
\end{array}
\right).$$ 
The elements of this square array are the Whitney numbers $W_{n,k}$, where $W_{n,k}$ gives the maximal number of pieces into which $n$-space is sliced by $k$ hyperplanes. This is \seqnum{A004070} in the OEIS. The inverse of this matrix is then an upper Hessenberg matrix. The number triangle that corresponds to the Whitney numbers, which begins
$$\left(
\begin{array}{cccccc}
 1 & 0 & 0 & 0 & 0 & 0 \\
 1 & 1 & 0 & 0 & 0 & 0 \\
 1 & 2 & 1 & 0 & 0 & 0 \\
 1 & 2 & 3 & 1 & 0 & 0 \\
 1 & 2 & 4 & 4 & 1 & 0 \\
 1 & 2 & 4 & 7 & 5 & 1 \\
\end{array}
\right),$$ is given by the Riordan array 
$$\left(\frac{1}{1-x}, x(1+x)\right).$$ 
This triangle thus has its bivariate generating function given by 
$$\frac{1}{(1-x)(1-xy-x^2y)}.$$ 
The reflection of this matrix, which begins 
$$\left(
\begin{array}{cccccc}
 1 & 0 & 0 & 0 & 0 & 0 \\
 1 & 1 & 0 & 0 & 0 & 0 \\
 1 & 2 & 1 & 0 & 0 & 0 \\
 1 & 3 & 2 & 1 & 0 & 0 \\
 1 & 4 & 4 & 2 & 1 & 0 \\
 1 & 5 & 7 & 4 & 2 & 1 \\
\end{array}
\right),$$ is the triangle associated with the row partial sum matrix. This has generating function 
$$\frac{1}{(1-xy)(1-x-x^2y)}.$$ This is \seqnum{A052509}, the knights-move Pascal triangle.
\section{Generalized Stirling numbers}
Because of their Hessenberg form, we can use the inverses of the row partial sums of Riordan arrays as production matrices. We take the example of the inverse of the row partial sum of Pascal's triangle. This inverse begins 
$$\left(
\begin{array}{cccccc}
 2 & -1 & 0 & 0 & 0 & 0 \\
 -2 & 3 & -1 & 0 & 0 & 0 \\
 2 & -5 & 4 & -1 & 0 & 0 \\
 -2 & 7 & -9 & 5 & -1 & 0 \\
 2 & -9 & 16 & -14 & 6 & -1 \\
 -1 & 5 & -10 & 10 & -5 & 1 \\
\end{array}
\right).$$ 
The matrix generated by this production matrix then begins 
$$\left(
\begin{array}{cccccc}
 1 & 0 & 0 & 0 & 0 & 0 \\
 2 & -1 & 0 & 0 & 0 & 0 \\
 6 & -5 & 1 & 0 & 0 & 0 \\
 24 & -26 & 9 & -1 & 0 & 0 \\
 120 & -154 & 71 & -14 & 1 & 0 \\
 720 & -1044 & 580 & -155 & 20 & -1 \\
\end{array}
\right).$$ 
This is a signed version of \seqnum{A049444}, the array of generalized Stirling numbers of the first kind. It is given by the exponential Riordan array 
$$\left[\frac{1}{(1-x)^2}, \ln(1-x)\right].$$ 
In fact, we have the following general result.
\begin{proposition} The inverse of the row partial sum of the matrix $\left(\binom{n}{k}r^{n-k}\right)=\left(\frac{1}{1-rx}, \frac{x}{1-rx}\right)$ is the production matrix of the exponential Riordan array 
$$\left[\frac{1}{(1-rx)^{\frac{r+1}{r}}}, \frac{1}{r} \ln(1-rx)\right].$$
\end{proposition}
The inverse of the matrix $$\left[\frac{1}{(1-x)^2}, \ln(1-x)\right]$$ begins
$$\left(
\begin{array}{cccccc}
 1 & 0 & 0 & 0 & 0 & 0 \\
 2 & -1 & 0 & 0 & 0 & 0 \\
 4 & -5 & 1 & 0 & 0 & 0 \\
 8 & -19 & 9 & -1 & 0 & 0 \\
 16 & -65 & 55 & -14 & 1 & 0 \\
 32 & -211 & 285 & -125 & 20 & -1 \\
\end{array}
\right).$$ This is a signed version of \seqnum{A143494}, the triangle of $2$-Stirling numbers of the second kind. As an exponential Riordan array, this is $\left[e^{2x}, 1-e^x\right]$. Its rows give the coefficients of the characteristic polynomials of the principal minors of the Hessenberg matrix. 

\section{A variant: phyllotaxis}
In this final section, we look at a variant of the row partial sum. The row partial sum of a Riordan array is obtained by multiplying it on the right by the matrix $\left(\frac{1}{1-x}, x\right)^T$. In this section, we consider the case of multiplying on the right by the matrix $\tilde{\Sigma}$ which begins 
$$\left(
\begin{array}{cccccc}
 1 & 1 & 1 & 1 & 1 & 1 \\
 0 & 1 & 0 & 0 & 0 & 0 \\
 0 & 0 & 1 & 0 & 0 & 0 \\
 0 & 0 & 0 & 1 & 0 & 0 \\
 0 & 0 & 0 & 0 & 1 & 0 \\
 0 & 0 & 0 & 0 & 0 & 1 \\
\end{array}
\right).$$ 
This is the transpose of the matrix which begins 
$$\left(
\begin{array}{cccccc}
 1 & 0 & 0 & 0 & 0 & 0 \\
 1 & 1 & 0 & 0 & 0 & 0 \\
 1 & 0 & 1 & 0 & 0 & 0 \\
 1 & 0 & 0 & 1 & 0 & 0 \\
 1 & 0 & 0 & 0 & 1 & 0 \\
 1 & 0 & 0 & 0 & 0 & 1 \\
\end{array}
\right).$$ 
This is an almost Riordan array \cite{Almost} (of first order). That is, the matrix obtained by deleting the first row and the first column is a Riordan array (in this case, this is simply the identity matrix $I=(1,x)$). 
Starting with the binomial matrix, upon multiplication on the right by $\tilde{\Sigma}$, we obtain the matrix that begins
$$\left(
\begin{array}{cccccc}
 1 & 1 & 1 & 1 & 1 & 1 \\
 1 & 2 & 1 & 1 & 1 & 1 \\
 1 & 3 & 2 & 1 & 1 & 1 \\
 1 & 4 & 4 & 2 & 1 & 1 \\
 1 & 5 & 7 & 5 & 2 & 1 \\
 1 & 6 & 11 & 11 & 6 & 2 \\
\end{array}
\right).$$ 
Taking the inverses of the $n \times n$ principal submatrices of this matrix, we obtain the following sequence of inverse matrices.
$$\left(
\begin{array}{c}
 1 \\
\end{array}
\right),\left(
\begin{array}{cc}
 2 & -1 \\
 -1 & 1 \\
\end{array}
\right),\left(
\begin{array}{ccc}
 1 & 1 & -1 \\
 -1 & 1 & 0 \\
 1 & -2 & 1 \\
\end{array}
\right),
\left(
\begin{array}{cccc}
 2 & -2 & 2 & -1 \\
 -1 & 1 & 0 & 0 \\
 1 & -2 & 1 & 0 \\
 -1 & 3 & -3 & 1 \\
\end{array}
\right),$$
$$\left(
\begin{array}{ccccc}
 1 & 2 & -4 & 3 & -1 \\
 -1 & 1 & 0 & 0 & 0 \\
 1 & -2 & 1 & 0 & 0 \\
 -1 & 3 & -3 & 1 & 0 \\
 1 & -4 & 6 & -4 & 1 \\
\end{array}
\right),\left(
\begin{array}{cccccc}
 2 & -3 & 6 & -7 & 4 & -1 \\
 -1 & 1 & 0 & 0 & 0 & 0 \\
 1 & -2 & 1 & 0 & 0 & 0 \\
 -1 & 3 & -3 & 1 & 0 & 0 \\
 1 & -4 & 6 & -4 & 1 & 0 \\
 -1 & 5 & -10 & 10 & -5 & 1 \\
\end{array}
\right),\ldots.$$ 
We collect the first rows of each of these matrices to form the matrix $A$ that begins
$$\left(
\begin{array}{ccccccc}
 1 & 0 & 0 & 0 & 0 & 0 & 0 \\
 2 & -1 & 0 & 0 & 0 & 0 & 0 \\
 1 & 1 & -1 & 0 & 0 & 0 & 0 \\
 2 & -2 & 2 & -1 & 0 & 0 & 0 \\
 1 & 2 & -4 & 3 & -1 & 0 & 0 \\
 2 & -3 & 6 & -7 & 4 & -1 & 0 \\
 1 & 3 & -9 & 13 & -11 & 5 & -1 \\
\end{array}
\right).$$ This is a signed variant of \seqnum{A122771}, which is related to the Fibonacci numbers and has applications in natural dynamical systems linked to phyllotaxis \cite{Kapraff}. This matrix has the factorization
$$\left(
\begin{array}{ccccccc}
 1 & 0 & 0 & 0 & 0 & 0 & 0 \\
 2 & -1 & 0 & 0 & 0 & 0 & 0 \\
 1 & 0 & -1 & 0 & 0 & 0 & 0 \\
 2 & -1 & 0 & -1 & 0 & 0 & 0 \\
 1 & 0 & -1 & 0 & -1 & 0 & 0 \\
 2 & -1 & 0 & -1 & 0 & -1 & 0 \\
 1 & 0 & -1 & 0 & -1 & 0 & -1 \\
\end{array}
\right) \cdot \left(
\begin{array}{ccccccc}
 1 & 0 & 0 & 0 & 0 & 0 & 0 \\
 0 & 1 & 0 & 0 & 0 & 0 & 0 \\
 0 & -1 & 1 & 0 & 0 & 0 & 0 \\
 0 & 1 & -2 & 1 & 0 & 0 & 0 \\
 0 & -1 & 3 & -3 & 1 & 0 & 0 \\
 0 & 1 & -4 & 6 & -4 & 1 & 0 \\
 0 & -1 & 5 & -10 & 10 & -5 & 1 \\
\end{array}
\right).$$ This is the product of an almost Riordan array and a Riordan array; the product is thus an almost Riordan array. Expressed in the notation of almost Riordan arrays \cite{Almost}, this is 
$$A=\left(\frac{1+2x}{1-x^2}; \frac{1}{(x-1)(1+x)^2}, \frac{x}{1+x}\right).$$ More interesting is the factorization
$$\left(
\begin{array}{cccccc}
 1 & 0 & 0 & 0 & 0 & 0 \\
 1 & -1 & 0 & 0 & 0 & 0 \\
 1 & -1 & -1 & 0 & 0 & 0 \\
 1 & -1 & -1 & -1 & 0 & 0 \\
 1 & -1 & -1 & -1 & -1 & 0 \\
 1 & -1 & -1 & -1 & -1 & -1 \\
\end{array}
\right) \cdot \left(
\begin{array}{cccccc}
 1 & 0 & 0 & 0 & 0 & 0 \\
 -1 & 1 & 0 & 0 & 0 & 0 \\
 1 & -2 & 1 & 0 & 0 & 0 \\
 -1 & 3 & -3 & 1 & 0 & 0 \\
 1 & -4 & 6 & -4 & 1 & 0 \\
 -1 & 5 & -10 & 10 & -5 & 1 \\
\end{array}
\right),$$  where we recognize the inverse of the binomial matrix on the right. Finally, we can represent the matrix $A$ as a triple product, which begins
$$\left(
\begin{array}{ccccc}
 1 & 0 & 0 & 0 & 0 \\
 1 & 1 & 0 & 0 & 0 \\
 1 & 2 & 1 & 0 & 0 \\
 1 & 3 & 3 & 1 & 0 \\
 1 & 4 & 6 & 4 & 1 \\
\end{array}
\right)\cdot \left(
\begin{array}{ccccc}
 1 & 0 & 0 & 0 & 0 \\
 0 & -1 & 0 & 0 & 0 \\
 0 & 1 & -1 & 0 & 0 \\
 0 & -1 & 2 & -1 & 0 \\
 0 & 1 & -3 & 3 & -1 \\
\end{array}
\right)\cdot \left(
\begin{array}{ccccc}
 1 & 0 & 0 & 0 & 0 \\
 -1 & 1 & 0 & 0 & 0 \\
 1 & -2 & 1 & 0 & 0 \\
 -1 & 3 & -3 & 1 & 0 \\
 1 & -4 & 6 & -4 & 1 \\
\end{array}
\right).$$ The diagonal sums of this matrix have generating function $\frac{1+2x-2x^2}{1-2x^2+x^3}$. They are the partial sums of the sequence of signed Fibonacci numbers that begins
$$1, 1, -2, 3, -5, 8, -13, 21, -34, 55, -89,\ldots.$$ 
The inverse matrix $A^{-1}$, which begins 
$$\left(
\begin{array}{cccccc}
 1 & 0 & 0 & 0 & 0 & 0 \\
 2 & -1 & 0 & 0 & 0 & 0 \\
 3 & -1 & -1 & 0 & 0 & 0 \\
 4 & 0 & -2 & -1 & 0 & 0 \\
 5 & 2 & -2 & -3 & -1 & 0 \\
 6 & 5 & 0 & -5 & -4 & -1 \\
\end{array}
\right),$$  has the factorization 
$$\left(
\begin{array}{cccccc}
 1 & 0 & 0 & 0 & 0 & 0 \\
 1 & 1 & 0 & 0 & 0 & 0 \\
 1 & 2 & 1 & 0 & 0 & 0 \\
 1 & 3 & 3 & 1 & 0 & 0 \\
 1 & 4 & 6 & 4 & 1 & 0 \\
 1 & 5 & 10 & 10 & 5 & 1 \\
\end{array}
\right)\cdot \left(
\begin{array}{cccccc}
 1 & 0 & 0 & 0 & 0 & 0 \\
 1 & -1 & 0 & 0 & 0 & 0 \\
 0 & 1 & -1 & 0 & 0 & 0 \\
 0 & 0 & 1 & -1 & 0 & 0 \\
 0 & 0 & 0 & 1 & -1 & 0 \\
 0 & 0 & 0 & 0 & 1 & -1 \\
\end{array}
\right).$$ 

The diagonal sums of this inverse matrix $A^{-1}$ are $F_n+1$ \seqnum{A001611}. 

The almost Riordan array $A$ has the property that its square $A^2$ is a Riordan array. We have
$$A^2=\left(\frac{1+x}{(1-x)(1-2x)}, \frac{x}{1+2x}\right)=B^{-1} \cdot \left(\frac{1}{1-2x}, x\right)\cdot B^{-1},$$ where $B$ is the binomial matrix, and the Riordan array $\left(\frac{1}{1-2x}, x\right)$ is the sequence array of $2^n$. 
\begin{proposition} The square of the almost Riordan array $\left(\frac{1+2x}{1-x^2}; \frac{1}{(x-1)(1+x)^2}, \frac{x}{1+x}\right)$ is the Riordan array $\left(\frac{1+x}{(1-x)(1+2x)}, \frac{x}{1+2x}\right)$. 
\end{proposition} 
\begin{proof}
By the theory of almost Riordan arrays, we have the factorization 
$$A =B \cdot C,$$ where $B=\left(\frac{1}{1-x^2}, \frac{x}{1+x}\right)$, a Riordan array, and where $C$ is the almost Riordan array $\left(\frac{1+x}{1-x}; -1,x\right)$. This array begins 
$$\left(
\begin{array}{cccccc}
 1 & 0 & 0 & 0 & 0 & 0 \\
 2 & -1 & 0 & 0 & 0 & 0 \\
 2 & 0 & -1 & 0 & 0 & 0 \\
 2 & 0 & 0 & -1 & 0 & 0 \\
 2 & 0 & 0 & 0 & -1 & 0 \\
 2 & 0 & 0 & 0 & 0 & -1 \\
\end{array}
\right).$$ Now we use the fact that $C$ is an involution in the group of almost Riordan arrays $(C^2=I)$ \cite{Slowik} to conclude the proof.
\end{proof}
We finish with the following conjecture. Regarding $\overline{A}$ as a production matrix, the matrix that it generates begins 
$$\left(
\begin{array}{cccccc}
 1 & 0 & 0 & 0 & 0 & 0 \\
 2 & -1 & 0 & 0 & 0 & 0 \\
 3 & -3 & 1 & 0 & 0 & 0 \\
 5 & -8 & 5 & -1 & 0 & 0 \\
 11 & -25 & 22 & -8 & 1 & 0 \\
 35 & -99 & 107 & -53 & 12 & -1 \\
\end{array}
\right),$$ with row sum polynomials $P_n(x)$ that begin
$$1, 2 - x, x^2 - 3x + 3, - x^3 + 5x^2 - 8x + 5, x^4 - 8x^3 + 22x^2 - 25x + 11,\ldots.$$

\begin{conjecture} We have 
$$P_n(x)= 1+(1-x) \sum_{i=0}^{n-1} (1-x)_i.$$ 
\end{conjecture} 

The production array of the inverse of the coefficient array of this family of polynomials begins 
$$\left(
\begin{array}{cccccc}
 2 & -1 & 0 & 0 & 0 & 0 \\
 1 & 1 & -1 & 0 & 0 & 0 \\
 1 & -1 & 2 & -1 & 0 & 0 \\
 1 & -1 & -1 & 3 & -1 & 0 \\
 1 & -1 & -1 & -1 & 4 & -1 \\
 1 & -1 & -1 & -1 & -1 & 5 \\
\end{array}
\right).$$

\bigskip
\hrule

\noindent 2010 {\it Mathematics Subject Classification}:
Primary 11B50; Secondary 05A15, 11B83, 11C20, 11Y55, 15B36.
\noindent \emph{Keywords:} Riordan array, generating function, partial sums, Hessenberg matrix

\bigskip
\hrule
\bigskip
\noindent (Concerned with sequences
\seqnum{A000045},
\seqnum{A001611},
\seqnum{A004070},
\seqnum{A007318}, 
\seqnum{A049444},
\seqnum{A052509},
\seqnum{A122771}, and
\seqnum{A143494}).

\end{document}